\documentclass[12pt]{article}
\usepackage{e-jc}
\usepackage{hyperref}
\usepackage{amsmath}
\usepackage{amssymb}
\usepackage{amsthm}
\usepackage{latexsym}
\usepackage{graphicx}

\DeclareMathOperator{\conv}{conv}

\DeclareMathOperator{\vol}{vol}

\DeclareMathOperator{\lcm}{lcm}

\theoremstyle{plain}
\newtheorem{theorem}{Theorem}

\newtheorem{lemma}[theorem]{Lemma}

\theoremstyle{definition}

\numberwithin{theorem}{section}

\makeatletter                          
\let\c@equation\c@theorem              
\makeatother                           

\makeatletter                          
\let\c@figure\c@theorem              
\makeatother                           

\providecommand{\floor}[1]{\left\lfloor#1\right\rfloor}
\newcommand{\R}{\mathbb{R}}
\newcommand{\Z}{\mathbb{Z}}
\newcommand{\Q}{\mathbb{Q}}

\long\def\symbolfootnote[#1]#2{\begingroup%
\def\thefootnote{\fnsymbol{footnote}}\footnote[#1]{#2}\endgroup}
\renewcommand{\mod}[1]{\,(\text{mod }#1)}

\providecommand{\D}{\mathcal{D}}
\renewcommand{\P}{\mathcal{P}}
\providecommand{\QQ}{\mathcal{Q}}

\title{A finite calculus approach to Ehrhart polynomials}

\author{Steven V Sam\\
\small Department of Mathematics\\[-0.8ex]
\small Massachusetts Institute of Technology\\
\small {\tt ssam@math.mit.edu}\\[-0.8ex]
\small \url{http://math.mit.edu/~ssam}\\
\and
Kevin M. Woods\\
\small Department of Mathematics\\[-0.8ex]
\small Oberlin College\\
\small {\tt kevin.woods@oberlin.edu}\\[-0.8ex]
\small \url{http://www.oberlin.edu/faculty/kwoods}
}

\date{\dateline{2009}{2010}\\
   \small Mathematics Subject Classification: 52C07}

\begin{document}

\maketitle

\begin{abstract}
  A rational polytope is the convex hull of a finite set of points in
  $\R^d$ with rational coordinates. Given a rational polytope $\P
  \subseteq \R^d$, Ehrhart proved that, for $t\in\Z_{\ge 0}$, the
  function $\#(t\P \cap \Z^d)$ agrees with a quasi-polynomial
  $L_\P(t)$, called the Ehrhart quasi-polynomial. The Ehrhart
  quasi-polynomial can be regarded as a discrete version of the volume
  of a polytope. We use that analogy to derive a new proof of
  Ehrhart's theorem. This proof also allows us to quickly prove two
  other facts about Ehrhart quasi-polynomials: McMullen's theorem about the
  periodicity of the individual coefficients of the quasi-polynomial
  and the Ehrhart--Macdonald theorem on reciprocity.

\end{abstract}

\section{Introduction.}
\label{section:introduction}
Let us first look at an (easy) example of computing a
\emph{volume}. Let $\Delta_{d}\subseteq\R^{d}$ be the convex hull of
the following $d+1$ points: the origin and the standard basis vectors
$e_i$, $1\le i\le d$. Let $t\Delta_{d}$ be its dilation by a factor of
$t$ (for nonnegative $t$). A straightforward way of computing the
volume of $t\Delta_{d}$ would be inductively in $d$, using the fact
that the $(d-1)$-dimensional cross section of $t\Delta_{d}$ at
$x_{d}=s$ is a dilated copy of $\Delta_{d-1}$:
\[
\vol\big(t\Delta_{d}\big) = \int_{0}^{t} \vol\big(s\Delta_{d-1}\big)\,
ds,
\]
and evaluating this iteratively gives us
$\vol\big(t\Delta_{d}\big)=t^{d}/d!$.

A generalization of volume is the \emph{Ehrhart (quasi-)polynomial},
which we define as follows. A $\emph{polytope}$, $\P$, is the convex
hull of finitely many points in $\R^{n}$, and the dimension,
$\dim(\P)$, of the polytope is the dimension of the affine hull of
$\P$. A {\it rational} (resp., {\it integral}) polytope is a polytope
all of whose vertices are rational (resp., integral). Given a polytope
$\P$ and a nonnegative $t$, let $t\P$ be the dilation of $\P$ by a
factor of $t$, and define the function $L_{\P} \colon \Z_{\ge
  0}\rightarrow \Z_{\ge 0}$ by
\[L_{\P}(t)=\#(t\P\cap\Z^{n}).\] Ehrhart proved
\cite{ehrhartpolynomial} that, if $\P$ is an integral polytope, then
$L_{\P}(t)$ is a polynomial of degree $\dim(\P)$. More generally, if
$\P$ is a rational polytope of dimension $d=\dim(\P)$, then
\[
  L_\P(t)=c_0(t) + c_1(t)t + \cdots + c_d(t)t^d\,,
\]
where the $c_i$ are periodic functions $\Z\rightarrow\Q$ (periodic
meaning that there exists an $s$ such that $c_{i}(t)=c_{i}(t+s)$ for
all $t$). Such functions are called \emph{quasi-polynomials}. Ehrhart
quasi-polynomials can be considered as a generalization of volume,
because, for full-dimensional $\P$, $c_{d}(t)$ is the constant
$\vol(\P)$. That is, $L_{\P}(t)$ is approximately
$\vol(t\P)=\vol(\P)t^{d}$, with lower degree terms correcting for the
fact that this is a discrete version of the volume computation.

Let us return to our polytope $t\Delta_{d}\subseteq\R^{d}$ and compute
its Ehrhart polynomial. For this example, our inductive approach to
computing volume works out well when translated to the discrete
problem. When $d=1$, $L_{\Delta_{1}}(t)=t+1$. We see that
\[L_{\Delta_{d}}(t)=\sum_{s=0}^{t}L_{\Delta_{d-1}}(s),\]
which we can prove by induction on $d$ and $t$ is
\[\frac{(t+1)(t+2)\cdots(t+d)}{d!}.\]
This calculation works out so well because expressions like the
falling factorial,
\[t^{\underline{d}}:=t(t-1)(t-2)\cdots(t-d+1),\] sum well. This is a
well-known fact from finite calculus \cite[Chapter
2]{concretemathematics}, and just as the polynomials $t^{d}$ form the
perfect basis of $\R[t]$ (as a vector space over $\R$) for integrating
because of the power rule, the polynomials $t^{\underline{d}}$ form the perfect
basis for summing, since
\begin{align} \label{powerrule}
\sum_{i=0}^{t}i^{\underline d} = \frac{1}{d+1}(t+1)^{\underline{d+1}}
\end{align}
(this fact can be proved quickly, by induction on $t$).

In Section \ref{section:proof}, we prove that this method of computing $L_{\P}(t)$ works
for any simplex (and hence for any polytope by triangulation). This
provides a new proof of Ehrhart's theorem that uses more minimal (but
less powerful) tools than other traditional proofs, such as proofs via
generating functions \cite{ehrhartpolynomial, stanley} or via
valuations \cite{mcmullenreciprocity}. Unlike these other proofs,
proving it for integral polytopes requires the full power of the proof
for rational polytopes. To prove it, we'll need a nice basis for the
vector space of quasi-polynomials of period $s$, which we shall
present in Section \ref{section:proof}.

This inductive computation of $L_{\P}(t)$ has two more desirable
outcomes: new and basic proofs of McMullen's theorem about periods of
the individual coefficients, $c_i(t)$, of the quasi-polynomial and of
Ehrhart--Macdonald reciprocity. We describe both of these results
below.

McMullen's theorem \cite[Theorem 4]{mcmullenreciprocity}, is as
follows. The \emph{$i$-index} of a rational polytope $\P$ is the
smallest number $s_i$ such that, for each $i$-dimensional face $F$ of
$\P$, the affine hull of $s_iF$ contains integer points. For this definition, we include $\mathcal P$ as a $d$-dimensional face of itself. Note that if $i \ge j$, then we must have
$s_i | s_j$: any $i$-dimensional face, $F$, contains $j$-dimensional faces, and so the affine hull of $s_{j}F$ contains integer points, though it may not be the smallest dilate to do so.

\begin{theorem}[McMullen's theorem] \label{ehrharttheorem} Given a rational polytope
  $\P\subseteq\R^n$, let $d=\dim(\P)$, and let
  \[
  L_\P(t)=\#(t\P\cap\Z^n)=c_0(t) + c_1(t)t + \cdots + c_d(t)t^d
  \]
  be the Ehrhart quasi-polynomial. Given $i$, with $0\le i\le d$, let
  $s_i$ be the $i$-index of $\P$. Then $s_i$ is a period of $c_i(t)$.
\end{theorem}

For example, let $\D$ be the smallest positive integer such that $\D
\P$ is integral. Then $s_i$ divides $\D$, for all $i$, and so $\D$ is
a period of each $c_i(t)$. If $\P$ is an integral polytope, then $\D=1$, and we recover that $L_{\P}(t)$ is actually a \emph{polynomial}. As another example, if $\P$ is full-dimensional then
the affine span of $\P$ is all of $\R^d$, and therefore $c_d(t)$ has
period 1. As mentioned, $c_d(t)$ is the constant which is the volume
of $\P$. McMullen's theorem is proven in Section \ref{section:proof}, concurrently with
Ehrhart's theorem.

Now we describe \emph{Ehrhart--Macdonald reciprocity}. Since the
function $L_\P(t)$ agrees with a quasi-polynomial $p(t)$ for all
positive integers, a natural question to ask is if $p(t)$ has any
meaning when $t$ is a negative integer, and indeed it does. Given a
polytope $\P$, let $\P^{\circ}$ be the relative interior of $\P$, that
is, the interior when considering $\P$ as a subset of its affine
hull. The number of integer points in $t\P^{\circ}$ is similarly
counted by $L_{\P^{\circ}}(t)$.

\begin{theorem}[Ehrhart--Macdonald reciprocity] \label{reciprocity}
  Let $\P$ be a rational polytope. Then  \[
  L_{\P^{\circ}}(t) = (-1)^{\dim(\P)} L_\P(-t)\,.
  \]
\end{theorem}

This statement was conjectured by Ehrhart, and he proved it in many
special cases. The general case was proven by Macdonald in
\cite{macdonald}. This will be proven in Section \ref{section:reciprocity}, using the following idea, which could be called a reciprocity theorem for finite calculus.

Suppose $f(s)$ is a quasi-polynomial, and suppose we are examining
$F(t)=\sum_{i=0}^{t}f(i)$. We will show in Section \ref{section:proof} that there is a
quasi-polynomial $p(t)$ such that $F(t)=p(t)$, for nonnegative
integers $t$. How about for negative integers? Certainly we can evaluate $p$ at a negative integer, $-t$, but we need to define what
\[F(-t)=\sum_{i=0}^{-t}f(i)
\]
should mean. Assuming that we want the summation rule
\[
\sum_{i=0}^{a}f(i)+\sum_{i=a+1}^{b}f(i) = \sum_{i=0}^{b}f(i)
\]
to hold for all integers $a$ and $b$, we must
have that
\[\sum_{i=0}^{-t}f(i) + \sum_{i=-t+1}^{0}f(i)=\sum_{i=0}^{0}f(i),\]
which means we should define
\[F(-t)=\sum_{i=0}^{-t}f(i) :=-\sum_{i=-t+1}^{-1}f(i).\]
Fortunately, when we plug in negative values, we still have $F(-t)=p(-t)$. This is the content of the following lemma, which we prove in Section  \ref{section:reciprocity}.

\begin{lemma}[Reciprocity for finite calculus] \label{NegativeLemma} Suppose that $f(i)$ is a
  quasi-polynomial in $i$. For nonnegative integers $t$ define the
  function
  \[
  F(t) = \sum_{i=0}^t f(i)\,.
  \]
  Then there is a quasi-polynomial $p(t)$ such that $F(t) = p(t)$ for
  all nonnegative integers $t$, and furthermore,
  \[
  p(-t) = -\sum_{i=-t+1}^{-1} f(i)
  \]
  for all $t>0$.
\end{lemma}

\section{Ehrhart's theorem and McMullen's theorem.}
\label{section:proof}

As mentioned in the introduction, ``discrete integration'' of
polynomials is made easy by using the basis $t^{\underline d}$ of
$\R[t]$. We will use the following generalization, which tells us how
to discretely integrate quasi-polynomials.
\begin{lemma} \label{LemmaSumQP} Let $f(t) = c_0(t) + c_1(t)t + \cdots
  + c_d(t)t^d$ be a quasi-polynomial of degree $d$, where $c_i(t)$ is
  a periodic function of period $s_i$, for each $i$. Define $F \colon
  \Z_{\ge 0}\rightarrow \Q$ by
\[
F(t) = \sum_{i=0}^{\floor{\frac{at}{b}}} f(i)\,,
\]
where $a,b \in \Z$ and $\floor{\cdot}$ is the greatest integer
function. Let $S_i = \frac{s_ib}{\gcd(s_i,a)}$. Then $F(t) = C_0(t) +
C_1(t)t + \cdots + C_{d+1}(t)t^{d+1}$ is a quasi-polynomial of degree
$d+1$. Furthermore, a period of $C_i(t)$ is $\lcm\{S_{d}, S_{d-1},
\ldots, S_{i}\}$, for $0 \le i \le d$, and $C_{d+1}$ has period 1.
\end{lemma}

Before we prove this lemma, let's look at an example. Suppose that
\[f(t)=\begin{cases} t/2 & \text{if $t$ even}\\0 & \text{if $t$ odd} \end{cases},\]
and we would like to evaluate the sum
\[F(t)=\sum_{i=0}^{\floor{3t/2}}f(i).\]
We have that $s_{1}=2$ and $s_{0}=1$, which give us $S_{1}=4$ and $S_{0}=2$. The lemma tells us that the period of the $t^{2}$ coefficient of $F(t)$ should be 1, the period of the $t^{1}$ coefficient should be $S_{1}=4$, and the period of the $t^{0}$ coefficient should be $\lcm\{S_{1},S_{0}\}=4$. Indeed,
\[F(t)= \sum_{i=0}^{\floor{3t/2}}f(i) = \sum_{j=0}^{\floor{3t/4}}j = \frac{1}{2}\left(\floor{3t/4}+1\right)^{\underline{2}}=
\begin{cases} \frac{9}{32}t^{2}+\frac{3}{8}t & \text{if }t\equiv 0\mod{4}\\
\frac{9}{32}t^{2}-\frac{3}{16}t-\frac{3}{32}  & \text{if }t\equiv 1\mod{4}\\
\frac{9}{32}t^{2}\hspace{0.41in} -\frac{1}{8}  & \text{if }t\equiv 2\mod{4}\\
\frac{9}{32}t^{2}+\frac{3}{16}t-\frac{3}{32}    & \text{if }t\equiv 3\mod{4}\\
\end{cases}.\]
Notice that this shows why taking the lcm of $S_{d},\ldots,S_{i}$ is necessary: $f(t)$ has periodicity only in the $t^{1}$ coefficient, but affects the period of both the $t^{1}$ and $t^{0}$ coefficients of $F(t)$.

\begin{proof}[Proof of \ref{LemmaSumQP}]
  Given $d$, $s$, and $j$, define the periodic function
\[
\chi_{s,j}(t)=\begin{cases} 1 & \text{if }t\equiv j\mod{s}\\
  0 & \text{otherwise}\end{cases}
\]
and the quasi-polynomial
\[
g_{d,s,j}(t) = \chi_{s,j}(t) \prod_{k=0}^{d-1} \left( \frac{t-j}{s}-k
\right)\,.
\]
For instance, in the preceding example, we had $f(t)=g_{1,2,0}(t)$.
For $t\equiv j\mod{s}$, substituting $t=ms+j$ gives us $g_{d,s,j}(ms+j)=m^{\underline{d}}$. This implies that, for a given $d$ and $s$, the set of $g_{d',s,j}$ such that $0\le d'\le d$ and $0\le j <s$ forms a basis (and, as we will see, a nice basis!) for the set of all quasi-polynomials of degree at most $d$ with period $s$. Writing our function $f(t)$ as a linear combination of such
quasi-polynomials (for various $d$, $s$, and $j$), it suffices to
prove that
\[
G_{d,s,j}(t) = \sum_{i=0}^{\floor{\frac{at}{b}}} g_{d,s,j}(i)
\]
is a quasi-polynomial of degree $d+1$ and period
$S=\frac{sb}{\gcd(s,a)}$ whose leading term has period 1
coefficient.

We have that, for any $k\in\Z_+$,
\begin{align*}
  \sum_{i=0}^k g_{d,s,j}(i) &=\sum_{m=0}^{\floor{\frac{k-j}{s}}} \!\! g_{d,s,j}(ms+j)\\
  &= \sum_{m=0}^{\floor{\frac{k-j}{s}}} \!\! m^{\underline d}\\
  &=\frac{1}{d+1}\left( \floor{\frac{k-j}{s}} + 1
    \right)^{\underline{d+1}},
\end{align*}
where the last line follows from \eqref{powerrule}, and so
\[
G_{d,s,j}(t) = \frac{1}{d+1} \left( \floor{ \frac{ \floor{
          \frac{at}{b}}-j}{s}} + 1 \right)^{\underline{d+1}}\,.
\]
One can check that this is a quasi-polynomial of period $S =
\frac{sb}{\gcd(s,a)}$ whose leading coefficient has period 1 by
substituting $t=mS+k$:
\begin{align*}
  G_{d,s,j}\left(m\frac{sb}{\gcd(s,a)}+k\right) &=
  \frac{1}{d+1} \left(\floor{\frac{\frac{ams}{\gcd(s,a)} + \floor{\frac{ak}{b}}-j}{s}} +  1 \right)^{\underline{d+1}}\\
  &=\frac{1}{d+1}\left(\frac{am}{\gcd(s,a)}+\floor{\frac{\floor{\frac{ak}{b}}-j}{s}}+1\right)^{\underline{d+1}},
\end{align*}
a polynomial in $m$ whose leading coefficient does not depend on $k$.
 The lemma follows.
\end{proof}

\begin{proof}[Proof of Ehrhart's Theorem and of \ref{ehrharttheorem}]
  We prove this by induction on $d$.  As the base case, consider
  $d=0$.  Then $\P$ is a point in $\Q^n$.  Let $\D$ be the smallest
  positive integer such that $\D \P$ is an integer point.  Then we see
  that
\[
L_\P(t)=c_0(t), \text{ where } c_0(t)=
\begin{cases}1 &\text{if }\D\big|t\\
0 & \text{otherwise}
\end{cases}.
\]

The base case follows. Now we assume that the theorem is
true for all $d'<d$.  We divide the proof into a number of steps:

~

\noindent \emph{1: Without loss of generality, we may assume that $\P$
  is full-dimensional, that is, $\dim(\P)=n$.}

Let $s'$ be the smallest positive integer such that the affine hull of
$s'\P$ contains integer points. Then we must have that $s'$ divides
each $s_i$. Let $V$ be the affine hull of $s'\P$. There is an affine
transformation $T \colon V \to \R^{\dim(\P)}$ that maps $V\cap\Z^n$
bijectively onto $\Z^{\dim(\P)}$. Let $\P'=T(s'\P)$. Then $\P'$ is a
full-dimensional polytope. If we can prove the theorem for $\P'$, it
will follow for $\P$, because
\[
L_\P(t)=\begin{cases}
  L_{\P'}\left(\frac{t}{s'}\right) & \text{if $s'$ divides $t$}\\
  0 & \text{otherwise}
\end{cases}.
\]

\noindent \emph{2: Without loss of generality, we may assume that
  $\P=\conv\{0, \QQ\}$, where $\QQ$ is a $(d-1)$-dimensional rational
  polytope.}

Assume that we have a general rational polytope $\P$, with
$\dim(\P)=d$. Without loss of generality, translate it by an integer
vector so that it does not contain the origin. We simply write
$L_\P(t)$ as sums and differences of polytopes of the form $\conv\{0,
\QQ\}$ (including lower dimensional $\QQ$), using inclusion-exclusion
to make sure that the intersections of faces are counted properly. The
exact form of this decomposition is not important for this proof, but
it will be important in Section \ref{section:reciprocity}, so we will
present it now. We examine two types of faces of $P$:
\begin{itemize}
\item The collection $\mathcal{F}_v$ of faces $F$ of $\P$ that are
  ``visible'': a facet (that is, a $(d-1)$-dimensional face) is said
  to be visible if, for all $a\in F$ and all $\lambda$ with
  $0<\lambda<1$, we have $\lambda a\notin \P$, and a lower dimensional
  face is visible if every facet that it is contained in is visible.

\item The collection $\mathcal{F}_h$ of faces $F$ of $\P$ that are
  ``hidden'': a facet is ``hidden'' if it is not visible, and a lower
  dimensional face is hidden if every facet that it is contained in is
  hidden.
\end{itemize}
Some lower dimensional faces may be neither visible nor hidden. For a
face $F$ of $\P$, let $\P_F=\conv(0,F)$. Then, using
inclusion-exclusion,
\begin{align} \label{indicators} L_\P(t) &= \sum_{F\in
    \mathcal{F}_h}(-1)^{d-1-\dim(F)} L_{\P_{F}}(t) - \sum_{F\in
    \mathcal{F}_v}(-1)^{d-1-\dim(F)} (L_{\P_{F}}(t) - L_F(t))\,.
\end{align}

\begin{figure}
  \begin{center}
    \includegraphics[height=.3\textheight]{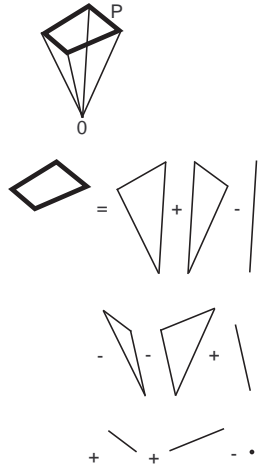}
  \end{center}
  \caption{Decomposition of a 2-dimensional polytope.}
  \label{polygondecomp}
\end{figure}

An example of this decomposition for a polygon is given in Figure
\ref{polygondecomp}. So as not to interrupt the flow of the proof, we
offer a proof of the correctness of \eqref{indicators} at the end of
the section.


For a given face $F$ of $\P$, the $i$-dimensional faces $F'$ of $\P_F$
are either faces of $\P$ or contain the origin. In either case, the
affine hull of $s_iF'$ contains integer points, so they meet the
conditions of the theorem. The theorem is true for the third piece of
the sum, $\sum_{F\in \mathcal{F}_v} (-1)^{d-\dim(F)} L_{F}(t)$, by the
induction hypothesis, since these are faces of smaller dimension than
$\P$.

~

\noindent \emph{3: Without loss of generality, we may assume that
  $\P=\conv\{0, \QQ\}$, where $\QQ$ is a $(d-1)$-dimensional rational
  polytope lying in the hyperplane $x_d = q$, where $q \in
  \Q_{>0}$.}

Perform a unimodular transformation such that this is true.

~

\noindent \emph{4: We prove the theorem for such a $\P$.}

We have $\P=\conv\{0, \QQ\}$, where $\QQ$ is a $(d-1)$-dimensional
rational polytope lying in the hyperplane $x_d = \frac{a}{b}$, where
$a,b\in\Z_{>0}$ and $\gcd(a,b)=1$.  Since faces of $\QQ$ are also
faces of $\P$, it follows that the affine hull of $s_iF$, where $F$ is
an $i$-dimensional face of $\QQ$, contains integer points. Let
$\bar{\QQ}=\frac{b}{a}\QQ$, lying in the hyperplane $x_{d}=1$. Then the affine hull of $s_{i}\frac{a}{b}\bar{F}$,
where $\bar{F}$ is an $i$-dimensional face of $\bar{\QQ}$, contains integer
points. We have that $t\P\cap\Z^d$ is the disjoint union
\[
\bigcup_{i=0}^{\floor{ \frac{ta}{b}}} i\bar{\QQ} \cap \Z^d\,,
\]
and so
\[
L_\P(t) = \sum_{i=0}^{\floor{\frac{ta}{b}}} L_{\bar{\QQ}}(i)\,.
\]
By Lemma \ref{LemmaSumQP}, this is a quasi-polynomial of degree $d$.
Furthermore, the $S_i$ in the statement of Lemma \ref{LemmaSumQP} are
given by
\[
S_i = \frac{\left(s_{i}\frac{a}{b}\right)b}{\gcd(s_{i}\frac{a}{b},a)} = \frac{as_i}{a}
= s_i\,.
\]
Since $s_d \big| s_{d-1} \big| \cdots \big| s_0$, $s_{i}=\lcm(s_{d},s_{d-1},\ldots,s_{i})$, and the coefficients of
$L_\P(t)$ have the desired periods. The theorem follows.
\end{proof}

More can be said about the period of $c_{d-1}(t)$. In this case,
$s_{d-1}$ is not only \emph{a} period but is guaranteed to be the
\emph{minimum} period. A proof of this fact along with a study of
maximal period behavior is given in \cite{quasiperiod} and relies only
on McMullen's theorem and Ehrhart--Macdonald reciprocity, which we
prove in the next section.

We also remark that, following the constant term through the
induction, we get for free another well-known fact: that the constant
term of the Ehrhart polynomial of a polytope is 1. More
precisely, the constant term of the Ehrhart polynomial for a polytopal
complex (open or closed) is equal to its Euler characteristic.

We close this section with a proof of \eqref{indicators}.

\begin{proof}[Proof of \eqref{indicators}] One can prove that this inclusion-exclusion is correct combinatorially, but the quickest proof to understand is topological. Let $\mathcal C=\bigcup_{F\in\mathcal F_{h}}F$ and $\mathcal C'=\bigcup_{F\in\mathcal F_{v}} F$. We only need to prove that the first
  sum in \eqref{indicators} counts each $x\in\conv\{0,\mathcal C\}$ exactly once, and that the second sum counts each $x\in\conv\{0, \mathcal C'\} \setminus \mathcal C'$ exactly once. Let's examine the first sum. It suffices to prove that each $x\in \mathcal C$ is counted correctly, as each $\lambda x\in\conv\{0,\mathcal C\}$ is counted identically to $x$.
  
 Assume for the moment that $x$ lies in the interior of $\mathcal C$, considered as a $(d-1)$-dimensional CW complex. Let $B$ be the intersection of $\mathcal C$ with the closure of a sufficiently small ball around $x$ (small enough so that $B$ only intersects faces $F$ that contain $x$). $B$ inherits a CW complex structure from $\mathcal C$. In particular, there is a one-to-one correspondence between $F\in\mathcal F_{h}$ that contain $x$ and cells of $B$ that are not contained in the boundary $\partial B$.  Therefore, in
  the first sum of \eqref{indicators}, the number of times the point $x$ is counted is
  \[n(x) = (-1)^{d-1} \sum_{F \owns x} (-1)^{\dim(F)} = (-1)^{d-1}
  \big(\chi(B) - \chi(\partial B)\big),\]
 where $\chi$ is the Euler characteristic (the alternating sum of the number of cells of each dimension). Since $B$ is contractible and $\partial B$ is homeomorphic to a $(d-2)$-sphere, $\chi(B)=1$ and $\chi(\partial B)= 1 + (-1)^{d-2}$. Hence
\[n(x) = (-1)^{d-1}\left[1-\left(1+(-1)^{d-2}\right)\right] = 1,\]
so $x$ is properly counted in the sum.

If $x$ is not in the interior of $\mathcal C$, notice that $x$ is counted exactly the same as any ``nearby'' point that is in the interior: the key is that faces on the boundary of $\mathcal C$ are not defined to be hidden faces in $\mathcal F_{h}$, because they are also contained in visible facets. Therefore these $x$ are also counted correctly. A similar argument shows that the
  second sum properly counts each $x\in\conv\{0, \mathcal C'\} \setminus \mathcal C'$.
 \end{proof}

\section{Reciprocity.}
\label{section:reciprocity}

In this section, we prove Theorem~\ref{reciprocity}. First we prove Lemma \ref{NegativeLemma}, which gives a reciprocity theorem for finite calculus.

\begin{proof}[Proof of Lemma \ref{NegativeLemma}]
  By Lemma \ref{LemmaSumQP}, there is a quasi-polynomial $p(t)$ such
  that $F(t)=p(t)$ for nonnegative integers $t$. Let $n$ be a fixed
  positive integer. Define
  \[
  C_n = \sum_{i=-n}^{-1}  f(i)\,,
  \]
  and for integers $t\ge -n$ define
  \[
  F_n(t) = -C_n + \sum_{i=-n}^t  f(i)\,.
  \]

  Using Lemma \ref{LemmaSumQP} (and reindexing as necessary), we see
  that there is a quasi-polynomial $p_n(t)$ such that $F_n(t)=p_n(t)$
  for all integers $t \ge -n$. But then we see that, for nonnegative
  integers $t$,
  \[
  p_n(t) = -C_n +  \sum_{i=-n}^t  f(i) = \sum_{i=0}^t f(i) = F(t) =
  p(t)\,.
  \]
  Because $p_n(t)$ and $p(t)$ agree for all nonnegative $t$, they must
  be identical as quasi-polynomials, and in particular
  \[
  p(-n)=-C_n +  \sum_{i=-n}^{-n}  f(i) =
  -\sum_{i=-n+1}^{-1}  f(i)\,,
  \]
  as desired.
\end{proof}

\begin{proof}[Proof of Theorem \ref{reciprocity}]
  Again, we induct on the dimension $d$ of the polytope. The inductive
  step will consist of two parts. First, assume $\P$ is a
  $d$-dimensional polytope that is the convex hull of the origin and
  $\QQ$, where $\QQ$ is a $(d-1)$-dimensional polytope. We shall first
  prove reciprocity for these types of polytopes. Second, having
  reciprocity for pyramids, we use the explicit inclusion-exclusion
  formula for the indicator functions of the ``exterior point
  triangulation'' given by \eqref{indicators} to show that reciprocity
  holds in general.

  Let $\QQ$ be a $(d-1)$-dimensional polytope in $\R^d$ contained in
  the hyperplane $x_d = \frac{a}{b}$ for nonzero, relatively prime integers $a$ and $b$,
  and let $\P$ be the pyramid $\conv\{0, \QQ\}$. As before, define
  $\bar{\QQ}=\frac{b}{a}\QQ$, lying in the hyperplane $x_{d}=1$. Let
  $f(i)$ give the number of lattice points in $i\bar{\QQ}$ and
  $f^{\circ}(i)$ give the number of lattice points in
  $i\bar{\QQ}^\circ$. By induction, we can assume that $f^\circ(i) =
  (-1)^{d-1}f(-i)$. So
  \[
  F(t)=\sum_{i=0}^{\left\lfloor \frac{ta}{b} \right\rfloor}  f(i)
  \]
  is the number of lattice points in $\P$, and
  \[
  F^\circ(t) =  \sum_{i=1}^{\left\lceil \frac{ta}{b}
    \right\rceil - 1}  f^\circ(i)
  \]
  is the number of lattice points in $\P^\circ$. Let $p(t)$ be the
  quasi-polynomial which corresponds to $F(t)$. By Lemma
  \ref{NegativeLemma}, $p(t)$ agrees with $F(t)$ for both positive and
  negative integers. Let $t' = \left\lceil \frac{ta}{b}
  \right\rceil$. Then
\begin{align*}
  (-1)^d p(-t) &= (-1)^d F(-t) \\
  &=(-1)^d \sum_{i=0}^{\floor{\frac{-ta}{b}}}
  f(i) \\
  &=(-1)^d \sum_{i=0}^{-t'} f(i) \\
  &=(-1)^{d+1} \sum_{i=-t' + 1}^{-1} \left[
    (-1)^{d-1}f^\circ(-i) \right] \\
  &=\sum_{i=1}^{t'-1} f^\circ(i)
  =F^\circ(t)\,.
\end{align*}

We now consider the case for general rational polytopes $\P$. As in
part 2 of the proof of Ehrhart's theorem, we write $\P$ as a sum and
difference of polytopes of the form $\conv\{0, \QQ\}$ and lower
dimensional polytopes. By equation~\eqref{indicators} and the
inductive hypothesis,
\begin{align*}
  L_\P(-t) &= \sum_{F\in \mathcal{F}_h}(-1)^{d-1-\dim(F)}L_{\P_{F}}(-t)
  -\sum_{F\in \mathcal{F}_v}(-1)^{d-1-\dim(F)} \big(L_{\P_{F}}(-t) -
  L_F(-t)\big)\\
  &= (-1)^d\left[\sum_{F\in \mathcal{F}_h} L_{\P_F^{\circ}}(t) - \sum_{F\in
    \mathcal{F}_v} \big(L_{\P_F^{\circ}}(t) - L_{F^{\circ}}(t)\big)\right]\,,
\end{align*}
and it is easy to see that the right hand side counts $(-1)^d$ times
the number of integer points in $t\P^{\circ}$, which finishes the
proof by induction.
\end{proof}

\section{Discussion.}

One might hope that this proof of Ehrhart's Theorem yields an efficient algorithm to compute Ehrhart polynomials inductively. To make this work, one must be able to efficiently compute a simple expression for sums like
\[\sum_{s=0}^{t}\floor{\frac{2s+3}{4}}\cdot \floor{\frac{3s+2}{5}}\]
(the summands are called \emph{step-polynomials} in \cite{steppolynomial}).  The only known way to compute such sums seems to be to first convert to a generating function using methods from \cite{steppolynomial} and then manipulate the generating function using the Barvinok Algorithm (which computes the generating function of the integer points in a polyhedron) and other related techniques, see \cite{BP99}. However, computing the sum in this way would be ill-advised, because the generating function techniques can compute the Ehrhart polynomial directly. Put another way, an elementary algorithm that, given a summation of a step-polynomial computes the sum as a new step-polynomial, would be interesting, because it would provide an alternative algorithm to Barvinok for answering questions about integer points in polytopes.

A second insight that this proof of Ehrhart's Theorem provides is the importance of picking nice bases in which to write Ehrhart polynomials and quasi-polynomials. Perhaps, rather than the standard basis for polynomials, $t^{d}$ for $d\ge 0$, a basis such as
\[L_{\Delta_{d}}(t)=\frac{(t+1)(t+2)\cdots(t+d)}{d!}\text{ for $d\ge 0$}\]
(which sums nicely) would be enlightening. A similar basis has already been studied: given $d$, $\binom{t+d-j}{d}$ for $0\le j\le d$ is a basis for polynomials of degree at most $d$. If we write
\[L_{\P}(t)=\sum_{j=0}^{d}h_{j}\binom{t+d-j}{d},\]
then the associated Hilbert series has the form
\[\sum_{s=0}^{\infty} L_{\P}(s)t^{s}=\frac{h_{0}+h_{1}t+\cdots+h_{d}t^{d}}{(1-t)^{d+1}}.\]
See, for example, Section 3.4 of \cite{continuous_discrete} for a discussion of this, including a proof of the fact that the $h_{j}$ are nonnegative. This basis has recently been used \cite{Bra} to study the roots of the Ehrhart polynomial, inspiring further study \cite{Pfe} of roots of polynomials whose coefficients are nonnegative in arbitrary bases. 

\section{Acknowledgements.}
The authors would like to thank Matthias Beck  and Timothy Chow for helpful discussions, and the anonymous referee for helping improve the exposition.


\bibliographystyle{amsplain}

\def\cprime{$'$} \def\cprime{$'$}
\providecommand{\bysame}{\leavevmode\hbox to3em{\hrulefill}\thinspace}
\providecommand{\MR}{\relax\ifhmode\unskip\space\fi MR }
\providecommand{\MRhref}[2]{%
  \href{http://www.ams.org/mathscinet-getitem?mr=#1}{#2} }
\providecommand{\href}[2]{#2}

\end{document}